\documentclass[12pt]{amsart}
\usepackage{dsfont}
\usepackage{bbm}
\usepackage{multirow}
\usepackage{placeins}
\usepackage{geometry}                
\usepackage{hyperref}

\newtheorem{theorem}{Theorem}[section]
\theoremstyle{plain}
\newtheorem{corollary}[theorem]{Corollary}

\newtheorem{prop}[theorem]{Proposition}
\newtheorem{remark}[theorem]{Remark}

\numberwithin{equation}{section}

\def\nn{\nonumber}
\def\ds{\displaystyle}
\def\tou{T}     
\def\tpq#1{\tou_#1(x;p,q)}
\def\nqexp{\mathrm{exp}_q}
\def\nexp#1{\mathrm{exp}_#1}

\def\Stir{S}
\def\stir{s}
\def\bell{B}          

\def\hpqStir#1#2{S(#1,#2;u,v)}

\def\nsb{\mathrm{nsb}}
\def\nse{\mathrm{nse}}
\def\rlm{\mathrm{rlm}}
\def\NSE{\mathrm{NSE}}
\def\RLM{\mathrm{RLM}}

\def\Spq{\Stir_{p,q}}

\def\slp{\mathrm{SLP}}
\def\aslp#1{|\slp_{#1}|}
\def\sone#1#2{\Big[\begin{matrix}#1 \\ #2 \end{matrix}\Big]}
\def\stwo#1#2{\Big\{\begin{matrix}#1 \\ #2 \end{matrix}\Big\}}
\def\sonet#1#2{\Big[\begin{matrix}#1 \\ #2 \end{matrix}\Big]}
\def\stwot#1#2{\Big\{\begin{matrix}#1 \\ #2 \end{matrix}\Big\}}

\def\SSP{\text{SSP}}
\def\LSP{\text{LSP}}
\def\SLP{\text{SLP}}
\def\LLP{\text{LLP}}

\begin{document}
\title{Study of the $p,q$-deformed Touchard polynomials}\author[O.~Herscovici]{Orli Herscovici}
\address{O.~Herscovici\\School of Mathematics,
Georgia Institute of Technology, Atlanta, 
GA 30332}
\email{orli.herscovici@gmail.com}

\begin{abstract}
A two-parameter deformation of the Touchard polynomials,
based on the NEXT $q$-exponential function
of Tsallis, defines two statistics on set partitions.  The generating function of classical Touchard polynomials is a composition of two exponential functions. By applying analysis of a combinatorial structure of the deformed exponential function, we establish explicit formulae for both statistics. Moreover, the explicit formulae for the deformed Touchard polynomials makes possible to evaluate coefficients of Taylor series expansion for wide variety of functions with different values of parameters $p$ and $q$.
\end{abstract}
\maketitle


\maketitle
\section{Introduction}
The  $p,q$-deformed Touchard polynomials $\tpq{n}$ for $n\in\mathbb{N}$ were defined in \cite{Herscovici2017} by means of the generating function
\begin{align}
\nexp{p}(x(\nexp{q}(t)-1))=\sum_{n=0}^\infty \tpq{n}
\frac{t^n}{n!}, \label{Eq1-tpq}
\end{align}
where $\nqexp(t)$ is the NEXT  $q$-exponential function given for $q\neq 1$ as follows
(see \cite{Tsallis1994})
\begin{align}
\nqexp(t)=(1+(1-q)t)^\frac{1}{1-q}. \label{DefNextqExp}
\end{align}
We extend this definition by assuming that for $q=1$ this deformed exponential function equals its limiting value $\lim_{q\rightarrow1}\nqexp(t)=e^t$.

In \cite{Herscovici2017} it was shown that $\tpq{n}$ is indeed an $n$th degree polynomial, and its
 coefficients were denoted  by $\Spq(n,k)$; that is
\begin{align}
\tpq{n}=\sum_{k=0}^n\Spq(n,k)x^k.\label{TpqSpq}
\end{align}

In the case that $p=q=1$, where one has
$$
e^{x(e^t-1)}=\sum_{n=0}^\infty T_n(x;1,1)\frac{t^n}{n!},
$$
the polynomials $T_n(x;1,1)$ are the classical Touchard polynomials, and the coefficients
$S_{1,1}(n,k)$ are the Stirling numbers of the second kind $\stwot{n}{k}$.
Recall that the Stirling numbers of the second kind  count the
number of partitions of $[n]$ into exactly $k$ disjoint non-empty sets. We can state this alternatively as follows: $S_{1,1}(n,k)=\stwot{n}{k}$ counts
the number
 of unordered partitions of $[n]$ into exactly $k$ unordered blocks. In the language of \cite{Motzkin1971},
 such  objects are called ``sets of sets.'' Thus, letting $\SSP_{n,k}$ denote the collection of all
 sets of sets of $[n]$ into $k$ blocks, we have $|\SSP_{n,k}|=S_{1,1}(n,k)=\stwot{n}{k}$.
 One could also consider ordered partitions of $[n]$ into $k$ non-empty unordered blocks; that is,   the order of the blocks is taken into
 account, but not the order of the elements inside the blocks. Alternatively, one could consider unordered partitions
 of $[n]$ into $k$ non-empty ordered blocks; that is,  the order of the elements in each block is taken into account, but not
 the order of the blocks. And finally, one could consider ordered partitions of $[n]$ into $k$ non-empty ordered blocks; that is,
 both the order of the blocks and the order of the elements within the blocks is taken into account.
 In the language of \cite{Motzkin1971}, these three types of objects are called
 ``lists of sets,'' ``sets of lists,'' and ``lists of lists'' respectively. We will denote these collections by
 $\LSP_{n,k},\SLP_{n,k}$ and $\LLP_{n,k}$ respectively.

The work of \cite{Herscovici2017} revealed an intimate connection between the above objects and the
deformed Touchard polynomials. We now describe that connection.
Consider an element $\pi\in\LLP_{n,k}$. For each of its blocks, define the ``standard opener'' of the block to be the smallest number in that block. Label the blocks by these numbers. In general, if we look at the blocks from left to right,
the labels will not be in increasing order. If the labels are in increasing order, we call the set of blocks a set of ``standard blocks.''
Imagine now that we can switch the order of the blocks so
that their labels will be in increasing order and thus the set of blocks will become a set of standard blocks. However, the
only way we are allowed to switch the positions of blocks is by moving  blocks to the
right. Let $\nsb(\pi)$ (``nonstandard blocks'') be the statistic denoting the number of blocks in $\pi$ that must be moved to the right to
 obtain a set of standard blocks. For example, let $n=8$ and $k=4$. If $\pi=32/681/57/4$, then
 $\nsb(\pi)=2$ because we need to move to the right the block 32 and the block 57, in order to obtain the standard
 set of blocks $618/32/4/57$.
Now consider any particular block of $\pi$. In general, the elements in the block are not in increasing order.
If they are, we call the set of elements of that block ``standard elements.'' Imagine now that we can switch the order of the elements
in the block so that they will be in increasing order and become a set of standard elements. However, the only way
we are allowed to switch the positions of the elements is by moving them to the right. Let $\text{nse}(\pi)$
(``nonstandard elements'') be the statistic denoting the number of elements from all the  blocks of $\pi$ that must be moved to the right in order that
all of its blocks will be sets of standard elements. Thus for $\pi$ given above, we have $\text{nse}(\pi)=3$, because
we need to move to the right the  numbers 3,6 and 8, in order to obtain 23/168/57/4.

From their definitions, it follows that $0\le \nsb(\pi)\le k-1$ and $0\le \nse(\pi)\le n-k$, for all $\pi\in\LLP_{n,k}$.
Define a generating function  $\hpqStir{n}{k}$ for the objects $\{\nsb(\pi),\nse(\pi);\pi\in \LLP_{n,k}\}$ by
\begin{align}
\hpqStir{n}{k}=\sum_{\pi\in \LLP_{n,k}}
u^{\nsb(\pi)}v^{\nse(\pi)}.\label{snkuv}
\end{align}
Here is the connection between $\LLP_{n,k}$ and the deformed Touchard polynomials that was shown
in \cite{Herscovici2017}.
\begin{theorem}(Theorem 3.4 in \cite{Herscovici2017})
For all integers $n,k$ with $1\le k\le n$,  we have
\begin{align}
\Spq(n,k)=\Stir(n,k;p-1,q-1). \label{spqnk}
\end{align}
\end{theorem}

This work continues the study of the $p,q$-deformed Touchard polynomials defined by \eqref{Eq1-tpq}. 
In \cite{Herscovici2017}, an explicit formulae for  $\Stir(n,k;u,0)$ was obtained by applying techniques of generating functions. However, no explicit formulae were given for $\Stir(n,k;0,v)$ or $\Stir(n,k;u,v)$. The reason is a more complicated recurrence relation involving the size $n$ of the set. In this work we establish those explicit formulae not by techniques of generating functions but rather by applying combinatorial analysis of the deformed exponential functions. These results  
are described in Section~2. In Section~3, we give an explicit expression for the $p,q$-deformed Touchard polynomials. Finally, in Section~4, we describe a natural connection of the statistic $\nse$ with the right-to-left minima statistic on permutations. 

\section{Sets of lists and lists of lists}
We begin this note by noting the special role of the coefficients $S_{p,q}(n,k)$ of the deformed Touchard
polynomials $T_n(x;p,q)$ in the particular cases that $p,q\in\{1,2\}$.
\begin{corollary}\label{cor}
\begin{equation}
\begin{aligned}
&\stwo{n}{k}=S_{1,1}(n,k)=|\SSP_{n,k}|;\ \ \  S_{2,1}(n,k)=|\LSP_{n,k}|;\\
&\ S_{1,2}(n,k)=|\SLP_{n,k}|;\ \ \
S_{2,2}(n,k)=|\LLP_{n,k}|.
\end{aligned}
\end{equation}
\end{corollary}
\begin{proof}[Combinatorial Proof]
The result follows immediately from \eqref{snkuv} and \eqref{spqnk}.
\end{proof}
\begin{remark}
It is easy to give a combinatorial proof that
$$
|\LSP_{n,k}|=k!\stwo{n}{k},\ \ |\LLP_{n,k}|=n!\binom {n-1}{k-1}\ \  \text{and}\ |\SLP_{n,k}|=\frac{n!}{k!}\binom {n-1}{k-1}.
$$
\end{remark}

More generally, if we set $u=p-1$ and $v=q-1$, then note that
$$
\begin{aligned}
&\left[\frac{t^n}{n!}x^ku^iv^j\right](\exp_{u+1}(x(\exp_{v+1}(t)-1))=[u^iv^j]S(n,k,u,v)=\\
&|\{\pi\in\LLP_{n,k}:\nsb(\pi)=i
\ \text{and}\ \nse(\pi)=j\}|.
\end{aligned}
$$
Note that we can also define $\nsb(\pi)$, for $\pi\in \LSP_{n,k}$, and we can also define $\nse(\pi)$, for
$\pi\in \SLP_{n,k}$, and that we have
$$
\begin{aligned}
&[u^iv^0]S(n,k,u,v)=[u^i]S(n,k,u,0)=
|\{\pi\in\LSP_{n,k}:\nsb(\pi)=i\}|;\\
&[u^0v^j]S(n,k,u,v)=[v^j]S(n,k,0,v)=
|\{\pi\in\SLP_{n,k}:\nsb(\pi)=j\}|.
\end{aligned}
$$
It was shown (see Theorem 3.5 in \cite{Herscovici2017}) that
\begin{align}
[u^iv^0]\Stir(n,k;u,v)=[u^i]\Stir(n,k;u,0)=\stwo{n}{k}\sone{k}{k-i},\label{old1}
\end{align}
where  $\sonet{n}{k}$ denotes unsigned Stirling numbers of the first kind.
Thus,
\begin{align}
|\{\pi\in\LSP_{n,k}:\nsb(\pi)=i\}|=\stwo{n}{k}\sone{k}{k-i}.\nn
\end{align}
\begin{remark}\label{remark-part-perm}
This formula has an additional aspect. Note that $|\{\pi\in\LSP_{n,n}:\nsb(\pi)=i\}|=\sonet{n}{n-i}$. Now converting $\pi\in\LSP_{n,n}$ to a standard block is equivalent to converting a single block of $n$ elements to standard elements. Thus we can translate the statistic $\nse$, defined initially on partitions of sets, into a statistic defined on permutations. There are $\sonet{n}{n-i}$ permutations $\sigma\in S_n$ with exactly $i$ elements that must be moved to the right in order to get the identity permutation.
\end{remark}

We will now calculate
$$
[u^0v^j]S(n,k,u,v)=[v^j]S(n,k,0,v)=
|\{\pi\in\SLP_{n,k}:\nse(\pi)=j\}|
$$
and     
$$
[u^iv^j]S(n,k,u,v)=
|\{\pi\in\LLP_{n,k}:\nsb(\pi)=i\ \text{and}\ \nse(\pi)=j\}|.
$$

Let us consider a combinatorial aspects of \eqref{DefNextqExp}.
For this purpose we replace $q-1=v$ as it was done in \cite{Herscovici2017}. Note that \eqref{DefNextqExp} has the following 
Taylor series expansion (see \cite{Borges1998})
\begin{align*}
\nexp{q}(t)=1+\sum_{n=1}^\infty Q_{n-1}(q)\frac{t^n}{n!},
\end{align*}
where $Q_n(q)=1\cdot q\cdot(2q-1)\cdot(3q-2)\cdot\ldots\cdot(nq-(n-1))$ with initial condition $Q_0(q)=1$.
Therefore, by applying our substitution, we obtain
\begin{align}
\nexp{q}(t)\Big|_{q-1=v}&=(1-vt)^\frac{1}{-v}=1+\sum_{n=1}^\infty Q_{n-1}(v+1)\frac{t^n}{n!}.\label{expv}
\end{align}
Let us consider now the coefficients of this exponential generating function for $n\geq 1$.
\begin{align}
\left[\frac{t^n}{n!}\right]\nexp{q}(t)\Big|_{q-1=v}&=Q_{n-1}(v+1)\\
&=(1+v)(1+2v)(1+3v)\cdot\ldots\cdot(1+(n-1)v)\\
&=\prod_{m=0}^{n-1}(mv+1)=\sum_{j=0}^{n-1}\sone{n}{n-j}v^j.\label
{qcoef}
\end{align}
From comparing \eqref{qcoef} with \eqref{old1} we can see that counting an $\nse$ in permutations of $[n]$ coincides with counting an $\nsb$ in partitions of $[n]$ with exactly $n$ blocks.
Therefore the coefficients of the
exponential generating function \eqref{expv} count the number of elements with respect to statistic $\nse$. To be more specific, let us denote by $\nse(\sigma_n,j)$ a number of permutations $\sigma_n$ of a set $[n]$ with exactly $j$  elements which should be moved to the right (see Remark~\ref{remark-part-perm}). Then for $0\leq j\leq n-1$, we have $\nse(\sigma_n,j)=\sonet{n}{n-j}$.

Now we can find the
explicit distribution of the statistic $\nse$ as a composition of exponential generating functions \cite[Theorem~3.3]{Aigner2007}. Composition of \eqref{expv} and the exponential function given by $e^{(1-vt)^\frac{1}{-v}}$ is another e.g.f. defining a distribution  of the number of 
elements which must be moved to the right (w.r.t. $\nse$) in partitions of a set $[n]$ into lists. The inner structure of the lists with respect to the statistic $\nse$ is described by \eqref{qcoef}. The outer structure --  a set partition -- is described by the Stirling numbers of the second kind. Therefore we can state the following theorem.

\begin{theorem}\label{Result}
\begin{equation}\label{result}
\begin{aligned}
&[u^0v^j]S(n,k,u,v)=[v^j](S(n,k,0,v)=
|\{\pi\in\SLP_{n,k}:\nse(\pi)=j\}| =\sone{n}{n-j}\stwo{n-j}{k}.
\end{aligned}
\end{equation}
\end{theorem}
\begin{corollary}
The average number of elements in partitions of $[n]$ into sets of lists in accordance with the statistic $\nse$ is given by
\begin{align*}
\frac{\ds{\sum_{j=0}^{n-1}j\sone{n}{n-j}\bell_{n-j}}}{\ds{\sum_{k=0}^n\sone{n}{k}\bell_k}},
\end{align*}
where $\bell_n$ is the $n$th Bell number.
\end{corollary}
\begin{proof}
Let us define a function $f_n(v)=
\sum_{k=1}^n\sum_{j=0}^{n-k}\sonet{n}{n-j}\stwot{n-j}{k}v^j$. After changing the summation order, we obtain $f_n(v)=
\sum_{j=0}^{n-1}\sonet{n}{n-j}\bell_{n-j}v^j$. 
Therefore the average number of elements w.r.t statistics $\nse$ is given by
\begin{align}
\frac{\partial_vf_n(v)\Big|_{v=1}}{\aslp{n}}
=\frac{\ds{\sum_{j=0}^{n-1}j\sone{n}{n-j}\bell_{n-j}}}{\ds{\sum_{k=0}^n\sone{n}{k}\bell_k}},
\end{align}
and the proof is complete.
\end{proof}
We have the following corollary concerning Stirling numbers of the second kind and unsigned Stirling numbers of the first kind.
\begin{corollary}
\begin{equation}\label{Stirling12}
\sum_{l=k}^n\sone{n}{l}\stwo{l}{k}=\frac{n!}{k!}\binom {n-1}{k-1}.
\end{equation}
\end{corollary}
\begin{proof}
The corollary follows immediately from the first equation in \eqref{result} and the formula
for $|\SLP_{n,k}|$ appearing in the remark following Corollary \ref{cor}.
\end{proof}
\bf\noindent Remark.\rm\ 
A completely combinatorical proof of \eqref{Stirling12} can be found in \cite{Herscovici2018a}.
By comparison, recall the classical formula concerning 
Stirling numbers of the second kind and (signed) Stirling numbers of the first kind 
\big($(-1)^{n-l}\sonet{n}{l}$\big):\newline
$\sum_{l=k}^n(-1)^{n-l}\sonet{n}{l}\stwot{l}{k}=\sum_{l=k}^n\stwot{n}{l}(-1)^{l-k}\sonet{l}{k}=
\delta_{nk}$.

Now we can state the results about the lists of lists.
\begin{theorem}\label{Result-LLP}
\begin{align}\label{coef-uivj}
[u^iv^j]S(n,k,u,v)&=
|\{\pi\in\LLP_{n,k}:\nsb(\pi)=i\ \text{and}\ \nse(\pi)=j\}|\nn\\
&=\sone{n}{n-j}\stwo{n-j}{k}\sone{k}{k-i}.
\end{align}
\end{theorem}

\begin{proof}
The proof can be easily obtained from the result \eqref{old1} and the Theorem~\ref{Result-LLP}  
as following. The number of partitions of $[n]$ into $k$ lists of sets w.r.t. the statistic $\nsb$ is given by $\stwot{n}{k}\sonet{k}{k-i}$ (see \eqref{old1}). The statistic $\nse$ is responsible for the arrangement of elements \emph{inside} the blocks and we already know that there are exactly $\sonet{n}{n-j}\stwot{n-j}{k}$ partitions of $[n]$ into $k$ blocks with $j$  
elements which must be moved to the right (see \eqref{result}). Therefore, by combining these results we obtain the statement \eqref{coef-uivj}.
\end{proof}

\section{$p,q$-deformed Touchard polynomials}
It was shown in \cite{Herscovici2017} that the coefficients of the $p,q$-deformed Touchard polynomials are closely related to the distribution of the statistics $\nsb$ and $\nse$ on set partitions \eqref{spqnk}. In previous section we obtained the explicit expression \eqref{coef-uivj} for coefficients of the terms $u^iv^j$. Now we can get an explicit expression for the deformed Touchard polynomials as polynomials in the variables $x$, $p$, and $q$. 
\begin{theorem}\label{th5}
The NEXT $p,q$-deformed Touchard polynomials $\tpq{n}$  defined by the generating function \eqref{Eq1-tpq} are polynomials in the variables $x$, $p$, and $q$ given by
\begin{align}
\tpq{n}=\sum_{k=0}^n\sum_{m=0}^{k-1}\sum_{\ell=0}^{n-k}\sum_{i=m}^{k-1}\sum_{j=\ell}^{n-k}\binom{i}{m}\binom{j}{\ell}(-1)^{m+\ell}
\stir(n,n-j)\Stir(n-j,k)\stir(k,k-i)p^mq^\ell x^k,\label{tpq-explicit}
\end{align}
where $\stir(i,j)$ and $\Stir(i,j)$ are the Stirling numbers of the first and second kind respectively.
\end{theorem}
\begin{proof}
From \eqref{TpqSpq} and \eqref{spqnk} we obtain
\begin{align*}
\tpq{n}=\sum_{k=0}^n\Spq(n,k)x^k=\sum_{k=0}^n\Stir(n,k;p-1,q-1)x^k.
\end{align*}
By applying  
\eqref{coef-uivj}, we get
\begin{align}
\tpq{n}=\sum_{k=0}^n\sum_{j=0}^{n-k}\sum_{i=0}^{k-1}\sone{n}{n-j}
\stwo{n-j}{k}\sone{k}{k-i}(p-1)^i(q-1)^jx^k. \label{eq100}
\end{align}
Binomial expansion of $(p-1)^i$ and $(q-1)^j$ leads to
\begin{align*}
T_n&(x;p,q)\nn\\
&=\sum_{k=0}^n\sum_{j=0}^{n-k}\sum_{i=0}^{k-1}\sum_{m=0}^i\sum_{\ell=0}^j\binom{i}{m}\binom{j}{l}(-1)^{(i-m)+(j-l)}\sone{n}{n-j}
\stwo{n-j}{k}\sone{k}{k-i}p^mq^\ell x^k.
\end{align*}
By applying the identity $\stir(i,j)=(-1)^{i-j}\sonet{i}{j}$ and changing the summation's order, we obtain
\begin{align*}
T_n&(x;p,q)\nn\\
&=\sum_{k=0}^n\sum_{m=0}^{k-1}\sum_{\ell=0}^{n-k}\sum_{i=m}^{k-1}\sum_{j=\ell}^{n-k}\binom{i}{m}\binom{j}{l}(-1)^{m+\ell}\stir(n,n-j)
\Stir(n-j,k)\stir(k,k-i)p^mq^\ell x^k,
\end{align*}
which proves the theorem.
\end{proof}
\begin{remark}
Note that for all $n\in\mathbb{N}$ it holds that $[x^0]\tpq{n}=0$ and $\tpq{0}=1$ for all values of $p$ and $q$.
\end{remark}

\begin{remark}
We already mentioned that the cases when $p,q\in\{1,2\}$ are related to combinatorial objects -- partitions of a set. However, the result of Theorem~\ref{th5} is significantly wider. It actually gives an explicit expression for the coefficients of expansion into Taylor series for a variety of functions defined for different values of the parameters $p$, $q$, and $x$.
\end{remark}

\section{Another connection to permutations}
We already show that statistic $\nse$, defined originally for set partitions, can be naturally  translated into permutations. Now we present a natural connection between this statistic and another statistic on permutations, namely, right-to-left minima. An element $\sigma_j$ of permutation $\sigma=\sigma_1\sigma_2\ldots\sigma_n$ is a right-to-left minimum of $\sigma$ if $\sigma_j<\sigma_k$ for all $k>j$. We can state the following proposition.
\begin{prop}
Given a permutation $\sigma\in S_n$. Let $\RLM(\sigma)$ be a set of the right-to-left minima of $\sigma$, $\NSE(\sigma)$ be a set of the elements that might be moved w.r.t. the statistic $\nse$, and $\nse(\sigma)=|\NSE(\sigma)|$, $\rlm(\sigma)=|\RLM(\sigma)|$. It holds that
\begin{align*}
&\NSE(\sigma)\cap\RLM(\sigma)=\emptyset,
&\nse(\sigma)+\rlm(\sigma)=n.
\end{align*}
\end{prop}
\begin{proof}
Given a permutation $\sigma=\sigma_1\sigma_2\ldots\sigma_n\in S_n$, then a statistic $\nse(\sigma)=|\{j:\sigma_j>\sigma_k\text{ for some } k>j\}|$. Let us sort the elements of $\sigma$ to two subsets: $\NSE(\sigma)$, containing all elements of $\sigma$ that must be moved to the right w.r.t. the statistic $\nse$, and $\RLM(\sigma)$, containing all elements that must be left in their relative to each other order. Now we are looking for the minimal element of $\sigma$. Let us assume that it is the element $\sigma_{j_1}$. Therefore the elements $\sigma_1,\ldots,\sigma_{j_1-1}$ go to the $\NSE(\sigma)$ and the element $\sigma_{j_1}$  goes to $\RLM(\sigma)$. Now we consider the rest of the permutation $\sigma$, namely, the sequence $\sigma_{j_1+1}\ldots\sigma_n$. Let us assume that the element $\sigma_{j_2}$ is its minimum. Therefore the elements $\sigma_{j_1+1},\ldots,\sigma_{j_2-1}$ go to the $\NSE(\sigma)$ and the element $\sigma_{j_2}$  goes to $\RLM(\sigma)$. We continue this process till all the elements of $\sigma$ are sorted. At the end of sorting each element of $\sigma$ belongs to only one of two subsets.
It follows from the sorting process that the subset $\RLM(\sigma)$ is the set of right-to-left minima of the permutation $\sigma$. 
Obviously, $|\NSE(\sigma)|=\nse(\sigma)$, and, with notation $\rlm(\sigma)=|\RLM(\sigma)|$, the proof is complete.
\end{proof}
\begin{remark}
There are three other related terms for permutations: left-to-right maxima, right-to-left maxima, and left-to-right minima, whose distributions are connected to each other. It is known that the number of permutations $\sigma\in S_n$ with exactly $k$ left-to-right maxima is $\sonet{n}{k}$ (see \cite[Corollary~3.42]{Bona2012}), which equals to the number of permutations of $\sigma\in S_n$ such that $\nse(\sigma)=n-k$.
\end{remark}

\section{Connection to the (reverse) Bessel polynomials}
In \cite{Herscovici2015}, by applying techniques of the umbral calculus, the classical Touchard polynomials were expressed in terms of the Bessel polynomials $p_n(x)$ considered by Carlitz in 1957 \cite{Carlitz1957} (see sequence A001497 at OEIS \cite{OEIS}). Those polynomials are defined by the following exponential generating function.
\begin{align}
e^{x\left(1-\sqrt{1-2t}\right)}=\sum_{n=0}^\infty p_n(x)\frac{t^n}{n!}.\label{bes1}
\end{align}
It worths to mention that the Bessel polynomials $p_n(x)$ are closely related to so called the reverse Bessel polynomials $\theta_n(x)$ (see \cite{Berg2006} and references therein). More precisely, one have
\begin{align*}
p_n(x)=x\theta_{n-1}(x).
\end{align*}
Let us consider now another special case of the $p,q$-Touchard polynomials, namely, the case $p=1$, $q=-1$. In this case, the generating function \eqref{Eq1-tpq} has the following form.
\begin{align}
e^{x\left(\sqrt{1+2t}-1\right)}=\sum_{n=0}^\infty\tou_n(x;1,-1)\frac{t^n}{n!}.
\label{bes2}
\end{align}

Comparing \eqref{bes1} and \eqref{bes2} we can easily obtain the following connection between the Bessel polynomials $p_n(x)$ and the $p,q$-Touchard polynomials.
\begin{theorem} For any nonnegative integer $n$, the Bessel polynomials $p_n(x)$ defined by \eqref{bes1} and the $p,q$-deformed Touchard polynomials defined by \eqref{Eq1-tpq} satisfy
\begin{align}
(-1)^np_n(-x)=\tou_n(x;1,-1),\label{bes3}
\end{align}
or, explicitly,
\begin{align}
(-1)^np_n(-x)=\sum_{k=0}^n\sum_{j=0}^{n-k}2^js(n,n-j)S(n-j,k)x^k.
\label{bes4}
\end{align}
\end{theorem}
\begin{proof}
The proof of \eqref{bes3} follows immediately from \eqref{bes1}-\eqref{bes2}.  To prove \eqref{bes4}, let us consider \eqref{eq100} with parameters $p=1$, $q=-1$, so we get
\begin{align*}
\tou_n(x;1,-1)&=\sum_{k=0}^n\sum_{j=0}^{n-k}(-1)^j\sone{n}{n-j}
\stwo{n-j}{k}2^jx^k\nn\\
&=\sum_{k=0}^n\sum_{j=0}^{n-k}2^js(n,n-j)S(n-j,k)x^k, 
\end{align*}
which completes the proof.
\end{proof}

\vspace{0.5cm}

\textbf{Acknowledgement}.
This research was supported by the Israel Science Foundation (grant No. 1692/17). The author thanks Ross G. Pinsky and Robert S. Maier for  comments and useful discussions.

\end{document}